\documentclass[11pt,a4paper]{amsart}

\usepackage{amsmath}
\usepackage{amssymb,amscd,amsxtra,calc}
\usepackage{mathrsfs}
\usepackage{mathtools}
\usepackage{enumitem}
\usepackage{tikz-cd}
\usepackage{array}
\usepackage{booktabs}
\usepackage[colorlinks,linkcolor=blue,anchorcolor=blue,citecolor=green,backref=none]{hyperref}

\theoremstyle{plain}
\newtheorem{thm}{Theorem}[section]

\newtheorem{theorem}[thm]{Theorem}
\newtheorem{proposition}[thm]{Proposition}
\newtheorem{lemma}[thm]{Lemma}
\newtheorem{corollary}[thm]{Corollary}

\theoremstyle{definition}
\newtheorem{definition}[thm]{Definition}

\newtheorem{question}[thm]{Question}
\newtheorem{remark}[thm]{Remark}
\newtheorem{conjecture}[thm]{Conjecture}

\newcommand{\Proj}{\mathrm{Proj}}

\newcommand{\rank}{\mathrm{rank}}

\newcommand{\End}{\mathrm{End}}
\newcommand{\codim}{\mathrm{codim}}
\newcommand{\Sing}{\mathrm{Sing}}

\title[Log Calabi--Yau structure for endomorphisms on $\mathbf{P}^n$]
{Log Calabi--Yau structure for endomorphisms on $\mathbf{P}^n$}

\author{Yujie Luo and Sheng Meng}
\date{}

\address{School of Mathematical Sciences, University of Science and Technology of China, People's Republic of China}

\email{yujieluo96@gmail.com}

\address{
    \textsc{School of Mathematical Sciences, Ministry of Education Key Laboratory of Mathematics and Engineering Applications \&
    Shanghai Key Laboratory of PMMP}\endgraf
    \textsc{East China Normal University, Shanghai 200241, China}\endgraf
}

\email{smeng@math.ecnu.edu.cn}



\begin{document}

\begin{abstract}
Let $f:\mathbf{P}^n\to\mathbf{P}^n$ be a $q$-polarized endomorphism, where $q>1$, and let $R_f$ be its ramification divisor. We study the singularities of the ramification pair $(\mathbf{P}^n,R_f)$. We show that, for a general $f$, the pair $(\mathbf{P}^n,R_f)$ is log canonical. When $n=2$, we prove that there exists an integer $s\geq1$ such that the log canonical threshold $\mathrm{lct}(\mathbf{P}^2;R_{f^s})\geq1/(q^s-1)$. The passage to an iterate is necessary in general, and the lower bound is optimal. In particular, $(\mathbf{P}^2,R_{f^s}/(q^s-1))$ is a log Calabi--Yau pair, completing the proof of Gongyo's conjecture for smooth projective surfaces.
\end{abstract}

\maketitle
\tableofcontents

\section{Introduction}

We work over the field of complex numbers.

\medskip

Let $f: X\to X$ be a surjective endomorphism of a smooth projective variety $X$. The ramification divisor formula
$$ K_X=f^*K_X+R_f $$
gives a natural effective divisor
$$ R_f:=\sum_P\bigl(f^*P-(f^{-1}(P))_{\mathrm{red}}\bigr), $$
where $P$ runs over the prime divisors on $X$.

More recently, the dynamical Iitaka theory introduced by Meng and Zhang \cite{MZ23} studies the asymptotic Iitaka dimension of effective divisors under iteration; see also \cite{MWY25} for its development for ramification divisors. Its main goal is to reduce questions about endomorphisms of general projective varieties to questions about polarized endomorphisms of simpler varieties.

The case of projective space is fundamental in the study of polarized endomorphisms, and we shall focus on it throughout. A finer classification of such endomorphisms in terms of the ramification divisor is then needed. One classical condition is postcritical finiteness: an endomorphism $f$ is called \emph{postcritically finite} (PCF) if its ramification locus $\mathrm{Supp}(R_f)$ is preperiodic. Recently, Gauthier, Taflin, and Vigny proved that, for every $q\geq2$, PCF endomorphisms are not Zariski dense in the parameter space of $q$-polarized endomorphisms of $\mathbf{P}^2$ \cite{GTV26}. This stands in sharp contrast to the classical one-dimensional picture, where PCF rational maps of any fixed degree are Zariski dense in the corresponding moduli space \cite{DeM18}. This sparsity result motivates a classification that detects finer geometric features of $R_f$.

It is natural to use the singularities of the ramification pair $(X,R_f)$ to study a polarized endomorphism $f$.

We first observe that the smoothness of $R_f$ for a general $f$ depends on the dimension. Denote by $\mathrm{End}_q(\mathbf{P}^n)$ the space of endomorphisms $f:\mathbf{P}^n\to\mathbf{P}^n$ with $f^*\mathcal{O}(1)\simeq\mathcal{O}(q)$.

\begin{theorem}
\label{thm: ramification-dimension-dichotomy}
Let $q\geq 2$.
\begin{enumerate}[label=\textup{(\arabic*)},leftmargin=2.5em]
\item If $1\leq n\leq 3$, then $R_f$ is smooth for a general $f\in\mathrm{End}_q(\mathbf{P}^n)$.
\item If $n\geq 4$, then $R_f$ is singular for every $f\in\mathrm{End}_q(\mathbf{P}^n)$.
\item For every $n\geq1$, the pair $(\mathbf{P}^n,R_f)$ is log canonical for a general $f\in\mathrm{End}_q(\mathbf{P}^n)$.
\end{enumerate}
\end{theorem}

Theorem~\ref{thm: ramification-dimension-dichotomy}(1) and (3) motivate a finer study of the singularities of ramification pairs associated with special endomorphisms, while Theorem~\ref{thm: ramification-dimension-dichotomy}(2) raises the question of how singular their ramification divisors must be in higher dimensions.
\begin{samepage}
This question is closely related to the following conjecture of Yoshinori Gongyo (cf. \cite[Conjecture~1.2]{BG17} and \cite[Conjecture~1.3]{Meng23}):

\begin{conjecture}[Gongyo]\label{conj: gongyo}
Let $f: X\to X$ be a $q$-polarized endomorphism of a smooth projective variety $X$. Then $\left(X,\frac{R_f}{q-1}\right)$ is log canonical after replacing $f$ by a suitable iterate.
\end{conjecture}
\end{samepage}

The second-named author proved Conjecture~\ref{conj: gongyo} for smooth projective surfaces except for the situation where $X\cong\mathbf{P}^2$ and $f$ admits no totally periodic curve (\cite[\S~4]{Meng23}). W.~Chang and D.-Q.~Zhang subsequently established the analogous result for normal projective surfaces \cite{CZ}, meeting similar difficulties. Our next theorem settles the remaining case $X=\mathbf{P}^2$, and hence completes the proof of Conjecture~\ref{conj: gongyo} for smooth projective surfaces. A pair $(X,\Delta)$ is said to be \emph{log Calabi--Yau} if it is log canonical and $K_X+\Delta\sim_{\mathbb{Q}}0$.

\begin{theorem}\label{thm: main projective plane}
Let $f: \mathbf{P}^2 \to \mathbf{P}^2$ be a $q$-polarized endomorphism. Then there exists an integer $s\geq1$ such that $\left(\mathbf{P}^2,\frac{R_{f^s}}{q^s-1}\right)$ is log canonical. In particular, it is a log Calabi--Yau pair.
\end{theorem}

The conclusion of Theorem~\ref{thm: main projective plane} cannot in general be strengthened to klt. Indeed, for the power map $f([x:y:z])=[x^q:y^q:z^q]$, one has $\frac{R_{f^s}}{q^s-1}=\{xyz=0\}$, so the normalized ramification pair is log canonical but not klt for every $s\geq1$. The passage to an iterate in Theorem~\ref{thm: main projective plane} is also necessary, as shown by \cite[Example~1.5]{LZ26}. Nevertheless, the first-named author and D.-Q.~Zhang confirmed Conjecture~\ref{conj: gongyo} for Galois self-covers of $\mathbf{P}^n$ in arbitrary dimensions, without passing to an iterate \cite[Theorem~1.4]{LZ26}.

We also record the part of the surface argument that extends to arbitrary dimension. For a $q$-polarized endomorphism $f:\mathbf{P}^n\to\mathbf{P}^n$, define, following \cite[Theorem~4.3]{Par11},
$$ \mathcal{E}_{\mathrm{Jac}}(f):=\bigcup_{\delta>0}\bigcap_{s\geq1}\left\{x\in\mathbf{P}^n:\mathrm{ord}_x(R_{f^s})\geq\delta q^s\right\}. $$
This set is totally invariant, meaning that $f^{-1}(\mathcal{E}_{\mathrm{Jac}}(f))=\mathcal{E}_{\mathrm{Jac}}(f)$.

\begin{theorem}\label{thm: fixed Jacobian exceptional set}
Let $f:\mathbf{P}^n\to\mathbf{P}^n$ be a $q$-polarized endomorphism, where $q>1$. The set $\mathcal{E}_{\mathrm{Jac}}(f)$ is a proper algebraic subset.
For every sufficiently large $s$, the pair $(\mathbf{P}^n,R_{f^s}/(q^s-1))$ is klt on the fixed open subset $\mathbf{P}^n\setminus\mathcal{E}_{\mathrm{Jac}}(f)$.
\end{theorem}

When $n=2$, one has $\mathcal{E}_{\mathrm{Jac}}(f)=\mathcal{E}_1(f)\cup\mathcal{E}_2(f)$ by \cite[Definition~3.15, Proposition~3.13, Theorem~4.1, and Proposition~5.1]{FJ03}, while \cite[Proposition~6.2]{FJ03} gives the uniform estimate used in the proof on its complement. Thus Theorem~\ref{thm: fixed Jacobian exceptional set} recovers the fixed-locus estimate used in the surface argument, while the valuative-tree method further distinguishes the exceptional curves from the exceptional points.
Since $\mathcal{E}_2(f)$ is finite and totally invariant, each of its points becomes totally invariant under a common iterate. Hence a recent result of W.~Chang and the first-named author \cite[Theorem~1.3]{CL26} gives the sharp bound $|\mathcal{E}_2(f)|\leq3$.

\medskip

We conclude this section by briefly explaining the main idea of the proof of Theorem~\ref{thm: main projective plane}.

\noindent\textbf{Sketch of the proof of Theorem~\ref{thm: main projective plane}.} Let $\mathcal{E}:=\mathcal{E}_1(f)\cup\mathcal{E}_2(f)$ be the exceptional set introduced in Definition~\ref{def: FJ-exceptional-sets}; roughly speaking, it is the locus where the local ramification multiplicities of the iterates of $f$ may have maximal growth. Put $U:=\mathbf{P}^2\setminus\mathcal{E}$. By \cite[Proposition~6.2]{FJ03}, recorded here as Proposition~\ref{prop: FJ-exceptional}(iii), there exist constants $C=C(f)>0$ and $0\leq\rho<q$ such that $\mathrm{ord}_x(R_{f^s})\leq C\rho^s$ for every $x\in U$ and every $s\geq1$. Together with the classical Izumi-type inequality, this shows that $(\mathbf{P}^2,\frac{R_{f^s}}{q^s-1})$ is klt on $U$ for all sufficiently large $s$. If $\mathcal{E}_1(f)\neq\varnothing$, the theorem follows from \cite[\S~4]{Meng23}; hence we may assume that $\mathcal{E}_1(f)=\varnothing$ and $\mathcal{E}=\mathcal{E}_2(f)$ is finite. 
After replacing $f$ by an iterate, we may assume that every point $x\in\mathcal{E}_2(f)$ is fixed and totally invariant, and the corresponding germ at $x$ is superattracting by \cite{FJ03}. Theorem~\ref{thm: balanced-local} then produces a primitive divisorial eigenvaluation $v_x$ for the common second iterate $f^2$. The assumption that the ambient variety is a surface is crucial at this point. The normalized Rees construction yields an equivariant birational model $\pi_x:Y_x\to\mathbf{P}^2$ extracting $v_x$, namely $v_x=\mathrm{ord}_{E_x}$ for a totally invariant exceptional prime divisor $E_x\subset Y_x$. Adjunction to $E_x$, followed by inversion of adjunction, gives local log canonicity near $x$. Taking a common multiple of the resulting iteration indices and combining these local conclusions with the estimate on $U$ completes the proof.

\vspace{2mm}

{\bf Acknowledgements.}
The second author is supported by the Science and Technology Commission of Shanghai Municipality (No. 22DZ2229014), the National Natural Science Foundation of China, and 
the Shanghai Pilot Program for Basic Research. 

\medskip

\section{Preliminaries}

We adopt the standard notation as in \cite{Ful}, \cite{Har}, \cite{Laz} and \cite{Mum74}.

\subsection{Valuations}

We recall the terminology from \cite[Section~2.1]{LLX}; see also \cite{ZS,JM12,BdFFU15}.

\begin{definition}\label{def: real-valuation}
Let $X$ be a reduced irreducible variety over $\mathbb{C}$, and let $K(X)$ be its function field. A \emph{real valuation} of $K(X)$ is a nonconstant map
$$ v: K(X)^*\longrightarrow\mathbb{R} $$
which is trivial on $\mathbb{C}^*$ and satisfies
\begin{enumerate}
    \item $v(fg)=v(f)+v(g)$,
    \item $v(f+g)\geq\min\{v(f),v(g)\}$.
\end{enumerate}
We set $v(0)=+\infty$. Its valuation ring and maximal ideal are
$$ \mathcal{O}_v:=\{f\in K(X):v(f)\geq 0\}, \qquad \mathfrak{m}_v:=\{f\in K(X):v(f)>0\}. $$
We say that $v$ is \emph{centered} at a scheme-theoretic point $\xi\in X$ if $\mathcal{O}_{X,\xi}\hookrightarrow\mathcal{O}_v$ is a local inclusion, and write $c_X(v)=\xi$. We denote by $\mathrm{Val}_X$ the set of real valuations admitting a center on $X$, and by $\mathrm{Val}_{X,x}$ the set of those centered at a closed point $x\in X$. For $v\in\mathrm{Val}_{X,x}$ and $t\geq 0$, its valuation ideal is $\mathfrak{a}_t(v):=\{f\in\mathcal{O}_{X,x}:v(f)\geq t\}.$ If $\mathfrak{a}\subseteq\mathcal{O}_{X,x}$ is a nonzero ideal, we set $$ v(\mathfrak{a}):=\min\{v(f):0\neq f\in\mathfrak{a}\}. $$
\end{definition}

\begin{definition}\label{def: divisorial-valuation}
Let $\pi: Y\to X$ be a proper birational morphism with $Y$ normal, and let $E\subset Y$ be a prime divisor. The order of vanishing along $E$ defines a valuation $\mathrm{ord}_E\in\mathrm{Val}_X$, whose center on $X$ is the generic point of $\pi(E)$. A valuation $v\in\mathrm{Val}_X$ is \emph{divisorial} if
$v=c\cdot\mathrm{ord}_E$ for some such $E$ and some $c\in\mathbb{R}_{>0}$.
For a smooth closed point $x\in X$ with maximal ideal $\mathfrak{m}_x$, we denote by $\mathrm{ord}_x$ the \emph{multiplicity valuation} defined by $\mathrm{ord}_x(f):=\max\{k\in\mathbb{Z}_{\geq 0}:f\in\mathfrak{m}_x^k\}$ for $0\neq f\in\mathcal{O}_{X,x}$. We set $\mathrm{ord}_x(0):=+\infty$ and extend $\mathrm{ord}_x$ to $K(X)^*$ by taking quotients.
\end{definition}

\begin{definition}\label{def: quasi-monomial-valuation}
A log-smooth model over $X$ consists of a proper birational morphism $\pi: Y\to X$, with $Y$ smooth, together with a reduced simple normal crossing divisor $E_Y=\sum_{k=1}^N E_k$ containing the exceptional locus of $\pi$. Let $\eta$ be the generic point of a codimension-$r$ stratum of $E_Y$, and let $E_{i_1},\ldots,E_{i_r}$ be the components of $E_Y$ containing $\eta$. Choose regular parameters $y_1,\ldots,y_r$ in $\mathcal{O}_{Y,\eta}$ such that $E_{i_j}=(y_j=0)$. Given weights $\alpha=(\alpha_1,\ldots,\alpha_r)\in\mathbb{R}_{>0}^r$ and a nonzero element $f\in\mathcal{O}_{Y,\eta}$, write in the completed local ring
$$ f=\sum_{\beta\in\mathbb{Z}_{\geq 0}^r}c_\beta y^\beta, \qquad y^\beta:=y_1^{\beta_1}\cdots y_r^{\beta_r}, $$
where each $c_\beta\in\widehat{\mathcal{O}}_{Y,\eta}$ is zero or a unit. The associated monomial valuation is defined by
$$ v_\alpha(f):=\min\{\langle\alpha,\beta\rangle:c_\beta\neq 0\}. $$
It extends to $K(X)^*$ by taking quotients. A valuation $v\in\mathrm{Val}_X$ is \emph{quasi-monomial} if $v=v_\alpha$ for some data as above.
\end{definition}

By \cite{ELS03}, a valuation $v\in\mathrm{Val}_X$ is quasi-monomial if and only if it satisfies Abhyankar's equality. Every divisorial valuation is quasi-monomial. At a smooth point $x$ of a surface, the space $\mathrm{Val}_{X,x}$ can be described, after normalizing by $v(\mathfrak{m}_x)=1$ and allowing centered semivaluations, in terms of the valuative tree introduced in Subsection~\ref{subsec: valuative tree}; see \cite{FJ04}. In higher dimensions, valuation spaces can be described as inverse limits of dual complexes \cite{JM12,BdFFU15}.

\subsection{Log canonical singularities and Izumi-type inequalities}

We start with a normal projective variety $X$ together with an effective Weil $\mathbb{Q}$-divisor $\Delta$ on $X$. $(X,\Delta)$ is called a pair if $K_X+\Delta$ is a $\mathbb{Q}$-Cartier divisor (see \cite{KM98}).

We first introduce the log discrepancies of quasi-monomial valuations following \cite[Definition~2.2]{LLX}.

\begin{definition}\label{defn: log discrepancy quasi-monomial valuations}
Let $E$ be a prime divisor on a normal variety $Y$ admitting a proper birational morphism $\pi: Y\to X$, and write $v_E=\mathrm{ord}_E$. The log discrepancy of $v_E$ with respect to $(X,\Delta)$ is
$$ A_{(X,\Delta)}(v_E):=1+\mathrm{ord}_E\bigl(K_Y-\pi^*(K_X+\Delta)\bigr). $$
For $c\in\mathbb{R}_{>0}$, we set $A_{(X,\Delta)}(c v_E):=cA_{(X,\Delta)}(v_E)$.

More generally, choose a representation of a quasi-monomial valuation $v_\alpha$ as in Definition~\ref{def: quasi-monomial-valuation} for which $E_Y$ also contains the strict transform of $\mathrm{Supp}(\Delta)$; such a representation exists after passing to a higher log-smooth model. If $E_{i_1},\ldots,E_{i_r}$ are the components containing $\eta$, we set
$$ A_{(X,\Delta)}(v_\alpha):=\sum_{j=1}^r\alpha_j A_{(X,\Delta)}(v_{E_{i_j}}). $$
This gives a well-defined, positively homogeneous function on the set of quasi-monomial valuations; see \cite{JM12,BdFFU15}. We write $A_X(v):=A_{(X,0)}(v)$ and $A_{(X,\Delta)}(E):=A_{(X,\Delta)}(v_E)$. If $X$ is $\mathbb{Q}$-factorial, then it follows from the definition that $$ A_{(X,\Delta)}(v)=A_{X}(v)-v(\Delta). $$
\end{definition}

For a klt pair, the log discrepancy function extends canonically from quasi-monomial valuations to all real valuations; see \cite[Definition~2.2(c)]{LLX}. In what follows, we only use log discrepancies of quasi-monomial valuations.

\medskip

We next recall \emph{log canonical} (lc) and \emph{Kawamata log terminal} (klt) singularities; see \cite[Definition~2.34]{KM98}.

\begin{definition}\label{defn: singularities and log discrepancies}
Let $(X,\Delta)$ be a pair, and let $\pi: Y\to X$ be a log resolution of $(X,\Delta)$ \cite{Hir64a,Hir64b}. The pair $(X,\Delta)$ is lc if $A_{(X,\Delta)}(v_E)\geq 0$ for every prime divisor $E$ on $Y$, and it is klt if $A_{(X,\Delta)}(v_E)>0$ for every such $E$. Equivalently, these inequalities hold for every divisorial valuation over $X$; hence the conditions are independent of the chosen log resolution.
\end{definition}

Suppose that $(X,\Delta)$ is log canonical on a neighborhood of a closed point $x\in X$, and let $D$ be a nonzero effective $\mathbb{Q}$-Cartier $\mathbb{Q}$-divisor defined near $x$. The \emph{log canonical threshold} of $D$ with respect to $(X,\Delta)$ at $x$ is denoted by
$$ \mathrm{lct}_x(X,\Delta;D):=\sup\{t\in\mathbb{R}_{\geq 0}:(X,\Delta+tD)\text{ is log canonical on a neighborhood of }x\}. $$
If $x\notin\mathrm{Supp}(D)$, then $\mathrm{lct}_x(X,\Delta;D)=+\infty$.
Throughout this paper, we only use the case $\Delta=0$, and for simplicity write $\mathrm{lct}_x(X;D):=\mathrm{lct}_x(X,0;D)$.

\begin{remark}\label{rem: equiv of all valuations}
Since log discrepancy is linear on every monomial cone and every divisorial valuation is quasi-monomial, a pair $(X,\Delta)$ is lc (resp. klt) if and only if $A_{(X,\Delta)}(v)\geq 0$ (resp. $>0$) for every quasi-monomial valuation $v$.
\end{remark}

\medskip

We recall the following classical Izumi-type inequality.

\begin{lemma}[{\cite[Proposition~9.5.13]{Laz04}}]\label{lem: classical izumi type inequality}
Let $x\in X$ be a smooth point, and let $D$ be an effective divisor on $X$. Then
$$ \mathrm{lct}_x(X;D)\geq \frac{1}{\mathrm{ord}_x(D)}. $$
Here $\mathrm{ord}_x$ is the multiplicity valuation from Definition~\ref{def: divisorial-valuation}.
\end{lemma}

\subsection{The valuative tree}\label{subsec: valuative tree}

We now specialize the preceding section to smooth surface germs.

\begin{definition}\label{defn: valuative tree}
Consider the local ring $R:=\mathcal{O}_{\mathbb{C}^2,0}$ with maximal ideal $\mathfrak{m}:=\mathfrak{m}_{\mathbb{C}^2,0}$. A \emph{centered normalized semivaluation} on $R$ is a map $\nu:R\longrightarrow\mathbb{R}_{\geq 0}\cup\{+\infty\}$ satisfying the same axioms as a real valuation in Definition~\ref{def: real-valuation}, except that $\nu(\phi)=+\infty$ is allowed for $0\neq\phi\in R$, and normalized by $\nu(\mathfrak{m}):=\min\{\nu(\phi):0\neq\phi\in\mathfrak{m}\}=1$. Let $\mathcal{V}$ be the collection of all centered normalized semivaluations on $R$. It carries a natural $\mathbb{R}$-tree structure, called the \emph{valuative tree} \cite{FJ04}.
\end{definition}

The finite-valued subspace
$$ \mathcal{V}^{\circ}:=\{\nu\in\mathcal{V}:\nu(\phi)<+\infty\text{ for every }0\neq\phi\in R\} $$
is a convex subtree of $\mathcal{V}$. Allowing the value $+\infty$ on a nonzero germ adjoins the curve valuations as ends of the valuative tree.
We order $\mathcal{V}$ pointwise: $\nu\leq\mu$ if $\nu(\phi)\leq\mu(\phi)$ for every $\phi\in R$. This is the rooted-tree order, and by \cite[Theorem~3.14]{FJ04}, $\mathcal{V}$ is complete with root $\mathrm{ord}_0$. We denote its unique segments by $[\nu,\mu]$; when $\nu\leq\mu$, one has $[\nu,\mu]=\{\lambda\in\mathcal{V}:\nu\leq\lambda\leq\mu\}$. An \emph{end} is a maximal point of $\mathcal{V}$, and a segment ending at an end will be called a \emph{terminal segment}.

In the notation of Definition~\ref{def: divisorial-valuation}, if $E$ is a prime exceptional divisor above the origin, its generic multiplicity and normalized divisorial valuation are $b_E:=\mathrm{ord}_E(\mathfrak{m})$ and $\nu_E:=b_E^{-1}\mathrm{ord}_E$.
Thus $b_E\nu_E=\mathrm{ord}_E$ is the associated primitive divisorial valuation; see \cite[Section~1.5.3]{FJ04}.

If $C$ is an irreducible, possibly formal, curve germ through the origin, with multiplicity $m(C)$, its normalized curve semivaluation is
$$ \nu_C(\phi):=\frac{i_0(C,\{\phi=0\})}{m(C)}, $$
where the intersection multiplicity $i_0(C,\{\phi=0\})$ is interpreted formally when necessary and is understood to be $+\infty$ when $\phi$ vanishes identically on $C$. The remaining ends, obtained as limits along infinite sequences of point blow-ups, are called infinitely singular valuations; see \cite[Sections~1.5.5 and~1.5.7]{FJ04}.

A quasi-monomial valuation $\nu$ is called \emph{irrational} if the skewness $\alpha(\nu)$ (see Proposition~\ref{prop: FJ-tree structure}(ii)) is irrational, or equivalently, $\nu$ is not divisorial.

\begin{proposition}\label{prop: FJ-tree structure}
The valuative tree has the following properties.
\begin{enumerate}[label=\textup{(\roman*)},leftmargin=2.2em]
\item \textup{(see \cite[Definition~2.23 and Proposition~3.20(ii)]{FJ04})} Its quasi-monomial points are precisely the nonends and are either divisorial or irrational. Its ends are precisely the curve valuations and the infinitely singular valuations.
\item \textup{(see \cite[Definition~3.23 and Theorem~3.26]{FJ04})} The skewness
$$ \alpha(\nu):=\sup_{0\neq\phi\in\mathfrak{m}}\frac{\nu(\phi)}{\mathrm{ord}_0(\phi)} $$
strictly increases along every nontrivial ordered segment. It is finite on the quasi-monomial subtree, and a quasi-monomial valuation is divisorial if and only if its skewness is rational. Consequently, the divisorial points are dense in every nondegenerate segment contained in the quasi-monomial subtree.
\item \textup{(see \cite[Remark~3.33]{FJ04})} For $t>0$, set $\mathfrak{a}_t(\nu):=\{\phi\in R:\nu(\phi)\geq t\}$. For every $\nu\in\mathcal{V}$, one has
\begin{equation}\label{eq: vol-skewness}
 \mathrm{vol}(\nu):=\limsup_{t\to\infty}\frac{\ell_R(R/\mathfrak{a}_t(\nu))}{t^2/2}=\alpha(\nu)^{-1}.
\end{equation}
Here and below, $(+\infty)^{-1}=0$.
\end{enumerate}
\end{proposition}

We next recall the local dynamics on $\mathcal{V}$. Let $f:(\mathbb{C}^2,0)\longrightarrow(\mathbb{C}^2,0)$
be a finite holomorphic fixed-point germ. Its local topological degree is
$$ d(f):=\ell_R(R/f^*\mathfrak{m}). $$
For $\nu\in\mathcal{V}$, define the following numerical invariants
$$ (f_*\nu)(\phi):=\nu(f^*\phi), \qquad c(f,\nu):=\nu(f^*\mathfrak{m}), \qquad f_\bullet\nu:=\frac{1}{c(f,\nu)}f_*\nu. $$
Since $f$ is finite, it contracts no curve, so $c(f,\nu)<+\infty$ and $f_\bullet\nu\in\mathcal{V}$ for every $\nu\in\mathcal{V}$. Moreover, $f_\bullet$ preserves divisorial, irrational, curve, and infinitely singular valuations \cite[Proposition~2.4]{FJ07}.
The attraction cocycle and functoriality give
$$ c(f^n,\nu)=\prod_{j=0}^{n-1}c(f,f_\bullet^j\nu), \qquad (f^n)_\bullet=f_\bullet^n; $$
see \cite[Section~2.1, especially Definitions~2.2--2.3 and equation~(2.1)]{FJ07}. The ordinary and asymptotic attraction rates are
$$ c(f):=c(f,\mathrm{ord}_0)=\max\{k\geq 1:f^*\mathfrak{m}\subseteq\mathfrak{m}^k\}, \qquad c_\infty(f):=\lim_{n\to\infty}c(f^n)^{1/n}. $$
The latter limit exists by supermultiplicativity \cite[Definition~4.1]{FJ07}.

Following \cite[Introduction and Section~2]{GR}, we call $f$ \emph{superattracting} if $c(f^n)\to+\infty$, equivalently if the differential $df_0$ is nilpotent. Following \cite[Theorem~4.2 and Definition~4.3]{FJ07} and \cite[Theorem~2.6 and Definition~2.7]{GR}, an \emph{eigenvaluation} of $f$ is a valuation $\nu_\star\in\mathcal{V}$ satisfying $f_\bullet\nu_\star=\nu_\star$ and $c(f,\nu_\star)=c_\infty(f)$, which is either quasi-monomial or an attracting end. Here an end $\nu_\star$ is attracting if there exists an $f_\bullet$-invariant terminal segment $[\nu_0,\nu_\star]$ such that $f_\bullet^n(\nu)\to\nu_\star$ for every $\nu\in[\nu_0,\nu_\star]$. Its multiplier is $c_\infty(f)$.

\begin{proposition}\label{prop: FJ-local dynamics}
Let $f:(\mathbb{C}^2,0)\to(\mathbb{C}^2,0)$ be a finite superattracting germ.
\begin{enumerate}[label=\textup{(\roman*)},leftmargin=2.2em]
\item \textup{(see \cite[Theorem~4.2]{FJ07})} There exists an eigenvaluation $\nu_\star\in\mathcal{V}$ such that
$$ f_*\nu_\star=c_\infty(f)\nu_\star. $$
\item \textup{(see \cite[Theorem~4.2, Proposition~3.4, and Proposition~5.2(i)]{FJ07})} If $\nu_\star$ is an end, then there exists $\nu_0<\nu_\star$ such that $c(f,\nu)=c_\infty(f)$ on $[\nu_0,\nu_\star]$, the map $f_\bullet$ is order preserving there, and
$$ f_\bullet\nu>\nu \qquad \text{for every }\nu\in[\nu_0,\nu_\star). $$
\item \textup{(see \cite[Theorem~5.1(ii) and Lemma~5.6]{FJ07})} If $\nu_\star$ is irrational, then there are a modification, a fixed point of the lifted germ, and local coordinates $(z,w)$ in which the lift is conjugate to
$$ \widehat f(z,w)=(z^aw^b,z^cw^d), \qquad M:=\begin{pmatrix}a&b\\ c&d\end{pmatrix}, $$
where $a,b,c,d\in\mathbb{Z}_{\geq 0}$ and $\det M\neq 0$. Moreover, $c_\infty(f)$ is the spectral radius of $M$, and the positive weight vector defining $\nu_\star$ is a $c_\infty(f)$-eigenvector of $M$.
\end{enumerate}
\end{proposition}

The alternatives needed in the local argument are supplied by the following refinement.

\begin{proposition}[Gignac--Ruggiero]\label{prop: GR trichotomy}
Let $f:(\mathbb{C}^2,0)\to(\mathbb{C}^2,0)$ be a finite superattracting germ. One of the following three alternatives occurs:
\begin{enumerate}[label=\textup{(\roman*)},leftmargin=2.2em]
\item $f$ has a unique eigenvaluation, which is an end of $\mathcal{V}$;
\item $f$ has a unique eigenvaluation, which is quasi-monomial;
\item there is a nondegenerate segment $I$ of quasi-monomial valuations fixed pointwise by $f_\bullet^2$.
\end{enumerate}
In the third alternative, if $I$ is maximal, then for every $\nu \in I$, we have
$$ c(f^2,\nu)=c_\infty(f)^2.$$
\end{proposition}

\begin{proof}
The trichotomy and the quasi-monomial nature of the fixed segment are given in \cite[Theorem~3.1 and Section~4]{GR}. The constancy of the attraction cocycle on a maximal fixed segment is proved in \cite[proof of Theorem~6.1, Case~4]{GR}.
\end{proof}

We finish by recalling the global attraction invariants that produce the local germs above.

\begin{definition}\label{def: global attraction rates}
Let $g:\mathbf{P}^2\to\mathbf{P}^2$ be a holomorphic endomorphism, and let $x\in\mathbf{P}^2$. The \emph{local attraction rate} of $g$ at $x$ and the \emph{asymptotic attraction rate} of $g$ at $x$ are, respectively,
$$ c(x,g):=\max\{m\geq 1:g^*\mathfrak{m}_{g(x)}\subseteq\mathfrak{m}_x^m\}, \qquad c_\infty(x,g):=\lim_{n\to\infty}c(x,g^n)^{1/n}. $$
The latter limit exists by \cite[Proposition~3.11]{FJ03}.
\end{definition}

If $x$ is fixed by $g$ and $f$ denotes the induced local germ, then $c(x,g^n)=c(f^n)$ for every $n$, and hence $c_\infty(x,g)=c_\infty(f)$.

\begin{definition}[{\cite[Definition~3.15]{FJ03}}]\label{def: FJ-exceptional-sets}
Let $g:\mathbf{P}^2\to\mathbf{P}^2$ be an endomorphism of algebraic degree $q\geq 2$, equivalently, a $q$-polarized endomorphism. By {\cite[Theorem~4.1]{FJ03}}, the first exceptional set $\mathcal{E}_1(g)$ is defined to be the union of the totally invariant curves of $g$, and the second exceptional set is
$$ \mathcal{E}_2(g):=\{x\in\mathbf{P}^2:c_\infty(x,g)=q\}. $$
It follows directly from the definition that $\{\mathcal{E}_i(g)\}_{i=1,2}$ does not change if $g$ is replaced by an iterate.
\end{definition}

\begin{proposition}\label{prop: FJ-exceptional}
Let $g:\mathbf{P}^2\to\mathbf{P}^2$ be as in Definition~\ref{def: FJ-exceptional-sets}. Then:
\begin{enumerate}[label=\textup{(\roman*)},leftmargin=2.2em]
\item \textup{(see \cite[Theorem~4.1]{FJ03})} $\mathcal{E}_1(g)$ consists of at most three lines in general position.
\item \textup{(see \cite[Propositions~3.13 and~5.1]{FJ03})} $\mathcal{E}_2(g)$ is finite and totally invariant, and every return germ at a periodic point of $\mathcal{E}_2(g)$ is superattracting.
\item \textup{(see \cite[Proposition~6.2]{FJ03})} There exist constants $C>0$ and $0\leq\rho<q$ such that $$\mathrm{ord}_x(R_{g^n})\leq C\rho^n$$ for every $x\notin\mathcal{E}_1(g)\cup\mathcal{E}_2(g)$.
\end{enumerate}
\end{proposition}

In the notation of \cite{Meng23}, the reduced divisor supported on $\mathcal{E}_1(g)$ is denoted by $T_g$. We record the following result.

\begin{theorem}[{\cite[Theorem~4.2]{Meng23}}]\label{thm: nonempty-E1}
Let $g:\mathbf{P}^2\to\mathbf{P}^2$ be a $q$-polarized endomorphism. If $\mathcal{E}_1(g)\neq\varnothing$, then $(\mathbf{P}^2,\frac{R_{g^s}}{q^s-1})$ is log Calabi--Yau for some integer $s\geq 1$.
\end{theorem}

In Proposition~\ref{prop: FJ-exceptional}(iii), the quantity $\mathrm{ord}_x(R_{g^n})$ is the multiplicity $\mu(x,Jg^n)$ appearing in \cite[Definition~3.3 and Proposition~6.2]{FJ03}, because on a smooth complex variety a local equation of the ramification divisor is the Jacobian determinant.

\section{Proof of Theorem~\ref{thm: ramification-dimension-dichotomy}}

\subsection{The universal Jacobian map}

Put
$$ B:=\End_q(\mathbf{P}^n),\qquad \mathfrak{X}:=B\times\mathbf{P}^n, $$
where $B$ is the smooth irreducible parameter space of endomorphisms of algebraic degree $q$. Let $\pi:\mathfrak{X}\to B$ be the first projection. The universal endomorphism over $B$ is
$$ \mathcal{F}:\mathfrak{X}\longrightarrow\mathfrak{X},\qquad (f,x)\longmapsto(f,f(x)), $$
and its relative ramification divisor satisfies
$$ \left.R_{\mathcal{F}}\right|_{\{f\}\times\mathbf{P}^n}=R_f. $$

Choose homogeneous coordinates $[x_0:\cdots:x_n]$ on $\mathbf{P}^n$. For $f=(f_0:\cdots:f_n)$, let
$$ J(f,x):=\left(\frac{\partial f_i}{\partial x_j}(x)\right)_{0\leq i,j\leq n}. $$
Euler's identity gives $J(f,x)\cdot x=qf(x)\neq0$. Scaling $x$ or the homogeneous tuple representing $f$ scales the entire matrix by a common nonzero constant. Hence
$$ J:\mathfrak{X}\longrightarrow M:=\mathbf{P}(\mathrm{Mat}_{n+1}),\qquad (f,x)\longmapsto[J(f,x)], $$
is a well-defined morphism, where $\mathrm{Mat}_{n+1}$ denotes the space of $(n+1)\times (n+1)$ matrices. Denote its restriction to $\{f\}\times\mathbf{P}^n$ by $J_f$.

\begin{lemma}\label{lem: universal Jacobian properties}
The morphism $J$ is smooth, and $J_f$ is finite onto its image for every $f\in B$.
\end{lemma}

\begin{proof}
Fix $(f,x)\in\mathfrak{X}$ and choose coordinates such that $x=[0:\cdots:0:1]$. Write $f_i=\sum_{|v|=q}c_{i,v}x^v$. The entries in the $j$-th column of $J(f,x)$ are
$$ c_{i,e_j+(q-1)e_n}\quad\text{if }j<n,\qquad q c_{i,qe_n}\quad\text{if }j=n. $$
Since $f(x)\neq0$, choose $i_0$ such that $c_{i_0,qe_n}\neq0$. Consider the affine chart $c_{i_0,qe_n}\neq0$ in the projective coefficient space and the affine chart of $M$ where the $(i_0,n)$-entry is nonzero. The coordinates on the latter chart are obtained by dividing every other matrix entry by the $(i_0,n)$-entry, and their pullbacks under $J$ are respectively
$$ \frac{c_{i,e_j+(q-1)e_n}}{q c_{i_0,qe_n}}\quad (j<n),
\qquad \frac{c_{i,qe_n}}{c_{i_0,qe_n}}\quad (j=n, i\neq i_0). $$
These are independent coefficient coordinates, so the differential of $J$ is surjective already in the coefficient directions. Since $\mathfrak{X}$ and $M$ are smooth, \cite[Chapter~III, Proposition~10.4]{Har} shows that $J$ is smooth.

The entries of $J_f$ are homogeneous forms of degree $q-1$ without a common zero. Therefore
$$ J_f^*\mathcal{O}_M(1)\simeq\mathcal{O}_{\mathbf{P}^n}(q-1). $$
Since this line bundle is ample, $J_f$ cannot contract a curve. Hence $J_f$ is quasi-finite, and projectivity shows that it is finite onto its image.
\end{proof}

Let
$$ D:=\{[A]\in M:\det A=0\},\qquad \Sigma:=\{[A]\in M:\rank A\leq n-1\}. $$
Then $D$ is the determinant hypersurface and $\Sigma=\Sing(D)$. For $n\geq2$, we note that $$\codim_M\Sigma=4.$$ Since the defining equation of $D$ pulls back to the homogeneous Jacobian determinant, one has the scheme-theoretic equality
$$ R_{\mathcal{F}}=J^*D. $$
Moreover, smoothness of $J$ and the Jacobian criterion give, as closed subsets,
$$ \Sing(R_{\mathcal{F}})=J^{-1}(\Sigma). $$
Indeed, if $h$ is a local equation of $D$ at $J(f,x)$, then
$$ d(h\circ J)_{(f,x)}=dh_{J(f,x)}\circ dJ_{(f,x)}. $$
Since $dJ_{(f,x)}$ is surjective, the left-hand side vanishes if and only if $dh_{J(f,x)}$ vanishes.

\subsection{Proof of the three assertions}

\begin{proof}[Proof of Theorem~\ref{thm: ramification-dimension-dichotomy}]
Suppose first that $1\leq n\leq3$. If $n=1$, then $\Sigma=\varnothing$ and $R_{\mathcal{F}}$ is smooth. If $2\leq n\leq3$, smooth pullback preserves codimension, and hence
$$ \dim\Sing(R_{\mathcal{F}})\leq\dim B+n-4<\dim B. $$
Since $\pi$ is projective, the subset $Z:=\pi(\Sing(R_{\mathcal{F}}))$ is closed and proper. In either case, there is a nonempty open subset $U\subseteq B$ such that
$$ R_{\mathcal{F},U}:=R_{\mathcal{F}}\times_B U $$
is smooth.

The induced morphism $R_{\mathcal{F},U}\to U$ is projective and flat. Indeed, $\mathfrak{X}\to B$ is smooth, and $R_{\mathcal{F}}$ is an effective Cartier divisor whose restriction to every fiber is cut out by the nonzero polynomial $\det J_f$. This polynomial is nonzero because the finite morphism $f$ is generically unramified in characteristic zero. The fiberwise Cartier criterion therefore gives flatness. Moreover, every fiber has degree $(n+1)(q-1)>0$ and is nonempty.

Let $C\subseteq R_{\mathcal{F},U}$ be the relative nonsmooth locus, which is closed because the smooth locus is open. The generic fiber is regular, since it is obtained from the smooth total space $R_{\mathcal{F},U}$ by localization, and hence it is smooth over the characteristic-zero function field of $U$. Thus $C$ does not meet the generic fiber. Since $R_{\mathcal{F},U}\to U$ is projective, the image of $C$ in $U$ is closed and proper. Removing this image, we obtain a smooth family. Therefore $R_f$ is smooth for a general $f$, proving \textup{(1)}.

Now assume $n\geq4$ and fix $f\in B$. By Lemma~\ref{lem: universal Jacobian properties}, the projective variety $Y_f:=J_f(\mathbf{P}^n)$ has dimension $n$. Since $\Sigma$ has codimension $4$ in $M$, the projective dimension theorem gives
$$ \dim(Y_f\cap\Sigma)\geq n-4. $$
Since $\Sigma=\Sing(D)$ and $R_f=J_f^*D$, the chain rule gives $J_f^{-1}(\Sigma)\subseteq\Sing(R_f)$. As $J_f$ is finite,
$$ \dim\Sing(R_f)\geq\dim J_f^{-1}(\Sigma)=\dim(Y_f\cap\Sigma)\geq n-4. $$
In particular, $R_f$ is singular. This proves \textup{(2)}.

It remains to prove \textup{(3)}. We use the standard fact that log discrepancies are preserved under smooth pullback. Consequently, log canonicity is preserved by smooth pullback and descends under smooth surjective morphisms.

Docampo's determinantal threshold formula~\cite[Theorem~D]{Docampo} implies that all log discrepancies of the affine determinant pair
$$ \left(\mathrm{Mat}_{n+1},\{A\in\mathrm{Mat}_{n+1}:\det A=0\}\right) $$
are nonnegative. Hence this pair is log canonical. The natural morphism
$$ \mathrm{Mat}_{n+1}\setminus\{0\}\longrightarrow M $$
is a smooth surjective $\mathbb{G}_m$-bundle, and the inverse image of $D$ is the affine determinant hypersurface with the origin removed. Thus $(M,D)$ is log canonical. The smoothness of $J$ and the equality $R_{\mathcal{F}}=J^*D$ then imply that $(\mathfrak{X},R_{\mathcal{F}})$ is log canonical.

By the Bertini theorem for pairs \cite[Lemma~5.17(1)]{KM98}, the restriction of a log canonical pair to a general member of a basepoint-free Cartier linear system is log canonical. Recall that $B$ is an open subset of the projective coefficient space. Choose a general complete flag in this projective space, ending at a point $f\in B$. Pulling back its successive hyperplane sections under $\pi$, at each step we obtain a general member of a basepoint-free Cartier linear system on the preceding inverse image. Repeated application of this result therefore shows that
$$ \left(\{f\}\times\mathbf{P}^n,\left.R_{\mathcal{F}}\right|_{\{f\}\times\mathbf{P}^n}\right) $$
is log canonical. Since $\left.R_{\mathcal{F}}\right|_{\{f\}\times\mathbf{P}^n}=R_f$, the pair $(\mathbf{P}^n,R_f)$ is log canonical for a general $f\in B$. This proves \textup{(3)}.
\end{proof}

\section{Log Calabi--Yau structure for \texorpdfstring{$\mathbf{P}^2$}{P2}}

\subsection{Existence of eigenvaluations for superattracting surface germs}

Let $(R,\mathfrak{m})$ be an $n$-dimensional regular local ring with residue field $\mathbb{C}$, and let $\varphi:R\to R$ be a finite local homomorphism of rank $d$. For a real valuation $v$ centered at $\mathfrak{m}$, set
$$ (\varphi_*v)(h):=v(\varphi(h)),\qquad c:=v(\varphi(\mathfrak{m})),\qquad \mu:=c^{-1}\varphi_*v. $$
For $t>0$, write $\mathfrak{a}_t(v):=\{h\in R:v(h)\geq t\}$ and
$$ \mathrm{vol}_R(v):=\limsup_{t\to\infty}\frac{\ell_R(R/\mathfrak{a}_t(v))}{t^n/n!}. $$
Consider the following $\mathfrak{m}$-filtrations (cf. \cite{BLQ24}):
$$
\begin{aligned}
\mathfrak{b}_{\bullet}&:=\{\mathfrak{b}_t\}_{t\in\mathbb{R}}=\{\varphi(\mathfrak{a}_t(\mu))R\}_{t\in\mathbb{R}},\\
\mathfrak{c}_{\bullet}&:=\{\mathfrak{c}_t\}_{t\in\mathbb{R}}=\{\mathfrak{a}_{ct}(v)\}_{t\in\mathbb{R}}.
\end{aligned}
$$
Let $\widetilde{\mathfrak{b}}_\bullet$ and $\widetilde{\mathfrak{c}}_\bullet$ denote the saturations of these $\mathfrak{m}$-filtrations in the sense of \cite[Definition~3.1]{BLQ24}. Note that if $\mathrm{vol}_R(v)>0$, then $\mathfrak{c}_\bullet=\mathfrak{a}_\bullet(c^{-1}v)$ is a positive-volume valuation filtration. By \cite[Lemma~3.20]{BLQ24}, $\mathfrak{c}_\bullet$ is saturated and $\widetilde{\mathfrak{c}}_\bullet=\mathfrak{c}_\bullet$.

\begin{lemma}\label{lem: local-volume}
With the above notation, we have
\begin{equation}\label{eq: volume-ineq}
 c^n\mathrm{vol}_R(v)\leq d\mathrm{vol}_R(\mu).
\end{equation}
Equality holds in \eqref{eq: volume-ineq} if and only if $\widetilde{\mathfrak{b}}_\bullet=\widetilde{\mathfrak{c}}_\bullet$.
\end{lemma}

\begin{proof}
If $h\in\mathfrak{a}_t(\mu)$, then $v(\varphi(h))=c\mu(h)\geq ct$, so $\mathfrak{b}_t\subseteq\mathfrak{c}_t$. Since $\varphi:R\to R$ is finite flat of rank $d$ by miracle flatness, the ideal inclusion gives
$$
\begin{aligned}
\ell_R\bigl(R/\mathfrak{a}_{ct}(v)\bigr)
&\leq\ell_R\bigl(R/\varphi(\mathfrak{a}_t(\mu))R\bigr)\\
&=d\,\ell_R\bigl(R/\mathfrak{a}_t(\mu)\bigr).
\end{aligned}
$$
Dividing by $(ct)^n/n!$ and taking the limit superior as $t\to\infty$ proves \eqref{eq: volume-ineq}.

The multiplicity limits of the two filtrations exist, and finite flatness gives
$$ e(\mathfrak{b}_\bullet)=d\,\mathrm{vol}_R(\mu),\qquad e(\mathfrak{c}_\bullet)=c^n\mathrm{vol}_R(v). $$
Consequently, equality in \eqref{eq: volume-ineq} is equivalent to $e(\mathfrak{b}_\bullet)=e(\mathfrak{c}_\bullet)$. Since $\mathfrak{b}_\bullet\subseteq\mathfrak{c}_\bullet$, the filtration version of Rees's theorem \cite[Theorem~1.4]{BLQ24} shows that this is equivalent to $\widetilde{\mathfrak{b}}_\bullet=\widetilde{\mathfrak{c}}_\bullet$.
\end{proof}

We now specialize to dimension two and use the notation of Subsection~\ref{subsec: valuative tree}. For $R=\mathcal{O}_{\mathbb{C}^2,0}$, the quantity $\mathrm{vol}_R(\nu)$ above is the volume $\mathrm{vol}(\nu)$ in \eqref{eq: vol-skewness}. If $f:(\mathbb{C}^2,0)\to(\mathbb{C}^2,0)$ is finite, then $f^*:R\to R$ is finite flat of rank $d(f)$ by miracle flatness. We now prove the key local statement.

\begin{theorem}\label{thm: balanced-local}
Let $f:(\mathbb{C}^2,0)\to(\mathbb{C}^2,0)$ be a finite superattracting germ. Assume that $d(f)=c_\infty(f)^2=q^2$, where $q>1$ is an integer. Then there exists a primitive divisorial valuation $v$ centered at the origin such that
\begin{equation}\label{eq: divisorial-eigen}
 (f^2)_*v=q^2v.
\end{equation}
\end{theorem}

\begin{proof}
We treat in order the three alternatives of Proposition~\ref{prop: GR trichotomy}.

\smallskip
\noindent\textbf{Case 1:} Proposition~\ref{prop: GR trichotomy}(i) applies, so $f$ has a unique eigenvaluation, which is an end of $\mathcal{V}$.
By Proposition~\ref{prop: FJ-local dynamics}(ii), there exists $\nu_0<\nu_\star$ such that $f_\bullet$ is order preserving on $[\nu_0,\nu_\star]$, one has $c(f,\nu)=c(f,\nu_\star)=c_\infty(f)=q$ and $f_\bullet\nu>\nu$ for every $\nu\in[\nu_0,\nu_\star)$.

Choose a quasi-monomial valuation $\nu$ strictly between $\nu_0$ and $\nu_\star$. Then $c(f,\nu)=q$ and $\alpha(f_\bullet\nu)>\alpha(\nu)$. Applying Lemma~\ref{lem: local-volume} with $d=q^2$ and $c=q$ gives $\mathrm{vol}(\nu)\leq\mathrm{vol}(f_\bullet\nu)$. By \eqref{eq: vol-skewness}, this is equivalent to $\alpha(f_\bullet\nu)\leq\alpha(\nu)$, a contradiction. Hence Case~(i) cannot occur.

\smallskip
\noindent\textbf{Case 2:} Proposition~\ref{prop: GR trichotomy}(ii) applies, so $f$ has a unique eigenvaluation, which is quasi-monomial.
By Proposition~\ref{prop: FJ-local dynamics}(i) and uniqueness, the distinguished eigenvaluation satisfies $f_*\nu_\star=c_{\infty}(f)\nu_\star=q\nu_\star$. It remains to show that $\nu_\star$ is divisorial. Suppose instead that it is irrational. By Proposition~\ref{prop: FJ-local dynamics}(iii), we have an integral exponent matrix $M$ whose spectral radius is $q$ and whose positive $q$-eigenvector gives the weights of $\nu_\star$.

Because $q\in\mathbb{Z}$, the matrix $M-qI$ has integral entries. If its kernel is one-dimensional, then it is defined over $\mathbb{Q}$, and the positive weight vector defining $\nu_\star$ is proportional to a positive rational vector. Thus $\nu_\star$ has rational weight ratio and is divisorial, contradicting its irrationality. If the kernel is two-dimensional, then $M=qI$, so every monomial valuation in the corresponding cone satisfies $f_*\nu=q\nu$. This gives a nondegenerate segment of eigenvaluations, contradicting uniqueness. Thus $\nu_\star$ is divisorial. Multiplying it by its generic multiplicity gives a primitive divisorial valuation $v$, which satisfies \eqref{eq: divisorial-eigen} by homogeneity and functoriality.

\smallskip
\noindent\textbf{Case 3:} Proposition~\ref{prop: GR trichotomy}(iii) applies, so there exists a nondegenerate segment $I$ of quasi-monomial valuations fixed pointwise by $f_\bullet^2$.
Let $I$ be a maximal nondegenerate segment fixed pointwise by $f_\bullet^2$. By Proposition~\ref{prop: FJ-tree structure}(ii), the divisorial points are dense in $I$; choose one of them, denoted by $\nu$. Since $f_\bullet^2\nu=\nu$, Proposition~\ref{prop: GR trichotomy}(iii) gives $(f^2)_*\nu=q^2\nu$. Multiplying by the generic multiplicity yields a primitive divisorial valuation $v$ satisfying \eqref{eq: divisorial-eigen}.
\end{proof}

\begin{remark}\label{rem: fixed-not-always-cinf}
The constancy statement in Proposition~\ref{prop: GR trichotomy}(iii) is essential. It is not true that every projectively fixed valuation of an arbitrary germ has multiplier $c_\infty$. For instance, the finite germ $(x,y)\mapsto(x^2,y^3)$ has two fixed coordinate curve valuations with raw multipliers $2$ and $3$, whereas its asymptotic attraction rate is $2$. What is used in Case~(iii) is the stronger, specific conclusion of the Gignac--Ruggiero fixed-segment alternative.
\end{remark}

\subsection{Equivariant extraction of a divisorial eigenvaluation}

Let $(R,\mathfrak{m})$ be a regular local ring. The Rees valuations of an $\mathfrak{m}$-primary ideal $I$ are the finitely many primitive divisorial valuations $v_i$ for which
$$ \overline{I^m}=\{h\in R:v_i(h)\geq m v_i(I)\text{ for every }i\}\qquad(m\geq 1). $$
If $v$ is the unique Rees valuation of $I$ and $b=v(I)$, then
\begin{equation}\label{eq: one-Rees}
 \overline{I^m}=\{h\in R:v(h)\geq mb\}
 \qquad(m\geq 1).
\end{equation}
When $\dim R=2$, Zariski's theory associates to every primitive divisorial valuation $v$ centered at $\mathfrak{m}$ a simple complete $\mathfrak{m}$-primary ideal whose unique Rees valuation is $v$; see \cite[Appendix~5]{ZS}. In dimension at least three, the analogous conclusion need not hold: even a special $*$-simple complete ideal associated with a single infinitely near point can have more than one Rees valuation \cite[Theorem~6.8]{HK}.

\begin{theorem}\label{thm: equivariant-extraction}
Let $X$ be a smooth projective variety of dimension $n\geq 2$, $f:X\to X$ a $q$-polarized endomorphism, and $p\in X$ a point satisfying $f^{-1}(p)=\{p\}$.
Suppose that $v$ is a primitive divisorial valuation centered at $p$ satisfying
\begin{equation}\label{eq: eigen-global-extraction}
 f_*v=qv.
\end{equation}
Assume moreover that $v$ is the unique Rees valuation of an $\mathfrak{m}_p$-primary ideal $I\subset\mathcal{O}_{X,p}$. Then there exist a normal $\mathbb{Q}$-factorial projective variety $Y$, a projective birational morphism $\pi:Y\to X$, and a finite $q$-polarized endomorphism $g:Y\to Y$ with the following properties:
\begin{enumerate}[label=\textup{(\roman*)},leftmargin=2.2em]
\item $\bigl(\pi^{-1}(p)\bigr)_{\mathrm{red}}=\mathrm{Exc}(\pi)=E$, where $E$ is a prime divisor satisfying $\mathrm{ord}_E=v$;
\item $\pi\circ g=f\circ\pi$;
\item $g^*E=qE$ as $\mathbb{Q}$-Cartier divisors.
\end{enumerate}
Moreover, if $b=v(I)$, then $bE$ is Cartier and $-E$ is $\pi$-ample as a $\mathbb{Q}$-Cartier divisor.
\end{theorem}

\begin{proof}
Set $R:=\mathcal{O}_{X,p}$, $\mathfrak{m}:=\mathfrak{m}_p$, and $b:=v(I)$. For every $h\in I$, the eigenrelation \eqref{eq: eigen-global-extraction} gives $v(f^*h)=(f_*v)(h)=qv(h)\geq qb$, so \eqref{eq: one-Rees} yields $f^*I\subseteq\overline{I^q}$. Moreover, $f^*I$ is $\mathfrak{m}$-primary because $f^{-1}(p)=\{p\}$.

Since $f^{-1}(p)=\{p\}$, the local homomorphism $f^*:R\to R$ is finite flat of rank $q^n$ by miracle flatness. Hence $\ell_R(R/f^*J)=q^n\ell_R(R/J)$ for every $\mathfrak{m}$-primary ideal $J\subset R$. Applying this to the powers of $J$ gives $e(f^*J)=q^ne(J)$, while $e(I^q)=q^ne(I)$. Since integral closure preserves Hilbert--Samuel multiplicity \cite{Rees},
$$ e(f^*I)=q^ne(I)=e(I^q)=e(\overline{I^q}). $$
Since $R$ is regular, it is formally equidimensional, so Rees's multiplicity theorem \cite{Rees} (see also \cite[Theorem~11.3.1]{HS}) applies to the inclusion $f^*I\subseteq\overline{I^q}$ and yields
\begin{equation}\label{eq: integral-closure-equality}
 \overline{f^*I}=\overline{I^q}.
\end{equation}
Taking powers of these integrally equivalent ideals and then taking integral closures gives
$$ \overline{(f^*I)^m}=\overline{I^{qm}}\qquad(m\geq 1). $$

Let $\mathcal{I}\subseteq\mathcal{O}_X$ be the coherent ideal sheaf whose stalk at $p$ is $I$ and which equals $\mathcal{O}_X$ on $X\setminus\{p\}$. Since $f^{-1}(p)=\{p\}$, the preceding equality globalizes stalkwise to
\begin{equation}\label{eq: sheaf-power-closure}
 \overline{(f^*\mathcal{I})^m}=\overline{\mathcal{I}^{qm}}
 \qquad(m\geq 1).
\end{equation}

Define the normalized blow-up
$$ Y:=\Proj_X\bigoplus_{m\geq 0}\overline{\mathcal{I}^m}, $$
and let $\pi:Y\to X$ be the structure morphism. Then $Y$ is the normalization of $\mathrm{Bl}_{\mathcal{I}}X$, and hence it is normal and projective.

On $Y$, we have $\mathcal{I}\mathcal{O}_Y=\mathcal{O}_Y(-D)$
for an effective Cartier divisor $D$ with $\mathrm{Supp}(D)=\pi^{-1}(p)_{\mathrm{red}}$. The prime components of $D$ correspond to the Rees valuations of $I$, and the coefficient of the component associated with a Rees valuation $w$ is $w(I)$. Since $v$ is the unique Rees valuation, there is a unique prime divisor $E$ satisfying $\mathrm{ord}_E=v$, and $D=bE$.
Consequently, $\pi^{-1}(p)_{\mathrm{red}}=\mathrm{Exc}(\pi)=E$,
and
\begin{equation}\label{eq: I-on-Y}
 \mathcal{I}\mathcal{O}_Y=\mathcal{O}_Y(-bE).
\end{equation}
Since $-D=-bE$ is $\pi$-ample, $bE$ is Cartier and $-E$ is $\pi$-ample as a $\mathbb{Q}$-Cartier divisor.

We claim that the variety $Y$ is $\mathbb{Q}$-factorial. Indeed, let $\Gamma\subset Y$ be a prime Weil divisor. If $\Gamma=E$, then $\Gamma$ is $\mathbb{Q}$-Cartier. Otherwise, $\Gamma_X:=\pi_*\Gamma$ is a prime divisor on the smooth variety $X$, and hence is Cartier. Since $E$ is the only $\pi$-exceptional prime divisor, there exists an integer $a$ such that $\pi^*\Gamma_X=\Gamma+aE$. It follows that $\Gamma=\pi^*\Gamma_X-aE$ is $\mathbb{Q}$-Cartier.

We next construct the lift of $f$. Let $\nu:Y\to\mathrm{Bl}_{\mathcal{I}}X$ be the normalization morphism, and consider the fibre product $X\times_{f,X}\mathrm{Bl}_{\mathcal{I}}X$, taken with respect to $f$ and the blow-up morphism. Its second projection to $\mathrm{Bl}_{\mathcal{I}}X$ is finite and surjective, being a base change of $f$. Since $f$ is finite flat by miracle flatness, flat base change for blow-ups gives a canonical isomorphism
$$ X\times_{f,X}\mathrm{Bl}_{\mathcal{I}}X\simeq\mathrm{Bl}_{f^*\mathcal{I}}X, $$
under which the first projection is the blow-up morphism. In particular, the fibre product is integral. Let $Y_f$ be its normalization. By the normalized Rees algebra description and \eqref{eq: sheaf-power-closure},
$$
\begin{aligned}
Y_f
&\simeq\Proj_X\bigoplus_{m\geq 0}\overline{(f^*\mathcal{I})^m}
\simeq\Proj_X\bigoplus_{m\geq 0}\overline{\mathcal{I}^{qm}}\\
&\simeq\Proj_X\bigoplus_{m\geq 0}\overline{\mathcal{I}^m}=Y.
\end{aligned}
$$
The last isomorphism follows from the invariance of relative $\Proj$ under passage to the $q$-th Veronese subalgebra; in particular, the resulting isomorphism $Y_f\simeq Y$ is over $X$.

The normalization map from $Y_f$ to the fibre product, followed by the second projection to $\mathrm{Bl}_{\mathcal{I}}X$, is finite and surjective. Since $Y_f$ is normal, this composite factors uniquely through $\nu$. Under the preceding identification over $X$, the resulting morphism defines a dominant endomorphism $g:Y\to Y$. The Cartesian square then gives the semiconjugacy $\pi\circ g=f\circ\pi$, summarized by
$$
\begin{tikzcd}[column sep=large,row sep=large]
Y \arrow[r,"g"] \arrow[d,"\pi"'] & Y \arrow[d,"\pi"] \\
X \arrow[r,"f"'] & X.
\end{tikzcd}
$$

Since $\nu\circ g$ is finite, $g$ is quasi-finite. It is also proper because $Y$ is projective, and hence finite. Since $g$ is dominant, it is surjective.

Using $\pi^{-1}(p)_{\mathrm{red}}=E$ and $f^{-1}(p)=\{p\}$, we obtain
$$ g^{-1}(E)=g^{-1}\bigl(\pi^{-1}(p)\bigr)=\pi^{-1}\bigl(f^{-1}(p)\bigr)=E $$
set-theoretically. Since $bE$ is Cartier, there is an integer $c>0$ such that $g^*(bE)=cE$ as Weil divisors. Pulling back \eqref{eq: I-on-Y} and taking the order along $E$ yields
$$ c=v(f^*I)=(f_*v)(I)=qv(I)=qb. $$
Therefore $g^*E=qE$ as $\mathbb{Q}$-Cartier divisors.

Finally, let $H$ be an ample Cartier divisor satisfying $f^*H\sim qH$. Since $-bE$ is $\pi$-ample, the Cartier divisor $L:=a\pi^*H-bE$ is ample for every sufficiently large integer $a$. Since $g^*(bE)=qbE$, we obtain
$$ g^*L\sim a\pi^*f^*H-g^*(bE)\sim aq\pi^*H-qbE=qL. $$
Thus $g$ is $q$-polarized.
\end{proof}

For surfaces, the Rees-valuation hypothesis in Theorem~\ref{thm: equivariant-extraction} is automatic.

\begin{corollary}\label{cor: equivariant-extraction-surface}
Let $X$ be a smooth projective surface, $f:X\to X$ a $q$-polarized endomorphism, $p\in X$ a point satisfying $f^{-1}(p)=\{p\}$, and $v$ a primitive divisorial valuation centered at $p$. Assume that
$$ f_*v=qv. $$
Then there exist a normal $\mathbb{Q}$-factorial projective surface $Y$, a projective birational morphism $\pi:Y\to X$, a finite $q$-polarized endomorphism $g:Y\to Y$, and a prime divisor $E\subset Y$ such that $\pi\circ g=f\circ\pi$, $(\pi^{-1}(p))_{\mathrm{red}}=\mathrm{Exc}(\pi)=E$, $\mathrm{ord}_E=v$, and
$$ g^{-1}(E)=E,\qquad g^*E=qE. $$
Moreover, $E$ is $\mathbb{Q}$-Cartier and $-E$ is $\pi$-ample.
\end{corollary}

\begin{proof}
By Zariski's theory, $v$ is the unique Rees valuation of a simple complete $\mathfrak{m}_p$-primary ideal; see \cite[Appendix~5]{ZS}. The assertion follows from Theorem~\ref{thm: equivariant-extraction}.
\end{proof}

\begin{remark}\label{rem: normalized-essential}
The normalized blow-up is essential. Equation \eqref{eq: integral-closure-equality} is an equality of integral closures, not necessarily an equality $f^*I=I^q$ of ideals. Consequently, the ordinary universal property of $\mathrm{Bl}_I X$ is not sufficient by itself.
\end{remark}

As a consequence, we prove the following:

\begin{proposition}\label{prop: crepant-descent}
Assume the hypotheses of Corollary~\ref{cor: equivariant-extraction-surface}. Then there exists an integer $s\geq 1$ such that $(X,\frac{R_{f^s}}{q^s-1})$ is log canonical in a neighborhood of $p$.
\end{proposition}

\begin{proof}
Let $\pi: Y\to X$, $g: Y\to Y$, and $E$ be as in Corollary~\ref{cor: equivariant-extraction-surface}. Since $X$ is smooth and $E$ is the only $\pi$-exceptional prime, there exists $a\in\mathbb{Q}$ such that
\begin{equation}\label{eq: KY-discrepancy}
 K_Y=\pi^*K_X+aE.
\end{equation}
The divisor $E$ is $\mathbb{Q}$-Cartier by Corollary~\ref{cor: equivariant-extraction-surface}, so $K_Y+E$ is $\mathbb{Q}$-Cartier. Moreover, $E$ is reduced and totally invariant under the polarized endomorphism $g$. Hence \cite[Corollary~3.3]{BH14} shows that $(Y,E)$ is log canonical.

The ramification index of $g$ along $E$ is $q$. Thus, for every $s\geq 1$, the divisor $B_s:=R_{g^s}-(q^s-1)E$ is effective and does not contain $E$, and
\begin{equation}\label{eq: logarithmic ramification along E}
 K_Y+E=(g^s)^*(K_Y+E)+B_s.
\end{equation}
Let $\nu:C\to E$ be the normalization and define $\Theta$ to be the $\mathbb{R}$-divisor satisfying $K_C+\Theta=\nu^*((K_Y+E)|_E)$ (see \cite[Page 1, second paragraph]{Kaw07}).
By adjunction (see the forward implication of \cite[Theorem]{Kaw07}), $(C,\Theta)$ is log canonical. The restriction of $g$ to $E$ lifts to a finite morphism $h: C\to C$. Since $g$ is $q$-polarized and $\dim Y=2$, one has $\deg g=q^2$. Moreover, $g^{-1}(E)=E$ and $g^*E=qE$, so the ramification index of the unique prime divisor $E$ lying over $E$ is $q$. The fundamental equality therefore gives
$$ q^2=\deg g=q[\mathbb{C}(E):h^*\mathbb{C}(E)]=q\deg h. $$
Thus $\deg h=q$.

By \eqref{eq: logarithmic ramification along E}, the divisor $B_s$ is $\mathbb{Q}$-Cartier. Since $B_s$ is effective and does not contain $E$, the divisor $\nu^*(B_s|_E)$ is effective. Using $\nu\circ h^s=(g^s|_E)\circ\nu$, restricting \eqref{eq: logarithmic ramification along E} to $E$, and then pulling back to $C$, we obtain
$$ K_C+\Theta=(h^s)^*(K_C+\Theta)+\nu^*(B_s|_E). $$
Comparing this equality with the Hurwitz formula $K_C=(h^s)^*K_C+R_{h^s}$ gives
$$ R_{\Theta,h^s}:=\Theta+R_{h^s}-(h^s)^*\Theta=\nu^*(B_s|_E)\geq 0. $$
Consequently, adjunction for the pair $(Y,E+B_s/(q^s-1))$ gives
$$ \nu^*((K_Y+E+\frac{B_s}{q^s-1})|_E)=K_C+\Theta+\frac{R_{\Theta,h^s}}{q^s-1}. $$
Since $h$ is $q$-polarized, \cite[Theorem~1.7]{CZ} (see also \cite[Remark~2.8]{BG17} when $\Theta=0$) implies that $(C,\Theta+\frac{R_{\Theta,h^s}}{q^s-1})$ is lc for some integer $s>0$.

Inversion of adjunction \cite[Theorem]{Kaw07} now shows that $(Y,E+\frac{B_s}{q^s-1})$, equivalently $(Y,\frac{R_{g^s}}{q^s-1})$, is log canonical in a neighborhood of $E$. From \eqref{eq: KY-discrepancy}, the semiconjugacy $\pi\circ g^s=f^s\circ\pi$, and $(g^s)^*E=q^sE$, we obtain $R_{g^s}=\pi^*R_{f^s}-a(q^s-1)E$.
Therefore
\begin{equation}\label{eq: crepant-identity}
 K_Y+\frac{R_{g^s}}{q^s-1}
 =\pi^*(K_X+\frac{R_{f^s}}{q^s-1}).
\end{equation}
This implies that the pair downstairs is log canonical in a neighborhood of $p$, proving the proposition.
\end{proof}

\begin{remark}
We provide an alternative proof of the proposition. Let $\pi:Y\to X$, $g:Y\to Y$, and $E$ be as in Corollary~\ref{cor: equivariant-extraction-surface}. Since $g^{-1}(E)=E$, the $\mathbb{Q}$-factorial surface $Y$ contains an inverse-periodic prime divisor. Hence \cite[Theorem~1.16(2)]{CZ}, applied to the $g$-pair $(Y,0)$, gives an integer $s\geq 1$ such that $(Y,\frac{R_{g^s}}{q^s-1})$
is log canonical. The crepant identity \eqref{eq: crepant-identity} then implies that $(X,\frac{R_{f^s}}{q^s-1})$ is log canonical, and hence it is log canonical in a neighborhood of $p$.
\end{remark}

\subsection{Proof of Theorem~\ref{thm: main projective plane}}

As mentioned in the introduction, it suffices to treat $X=\mathbf{P}^2$ in order to prove Gongyo's conjecture for smooth projective surfaces.

\begin{proof}
\noindent\textbf{Step 1.}
Let $H$ be the divisor of a line in $\mathbf{P}^2$. For every $s\geq 1$, the ramification divisor formula gives
$$ R_{f^s}\sim K_{\mathbf{P}^2}-(f^s)^*K_{\mathbf{P}^2}=-3H+3q^sH=3(q^s-1)H. $$
Consequently,
$$ K_{\mathbf{P}^2}+\frac{R_{f^s}}{q^s-1}\sim_\mathbb{Q} 0. $$
Thus it remains to find an iterate for which $(\mathbf{P}^2,\frac{R_{f^s}}{q^s-1})$ is log canonical.

\smallskip
\noindent\textbf{Step 2.}
Recall $\mathcal{E}_1(f)$ and $\mathcal{E}_2(f)$ from Definition~\ref{def: FJ-exceptional-sets}. If $\mathcal{E}_1(f)\neq\varnothing$, the result follows from Theorem~\ref{thm: nonempty-E1}. We may therefore assume that $\mathcal{E}_1(f)=\varnothing$.
Proposition~\ref{prop: FJ-exceptional}(iii) gives constants $C>0$ and $\rho<q$ such that
$$ \mathrm{ord}_x(R_{f^s})\leq C\rho^s $$
for every $s\geq 1$ and every closed point $x\notin\mathcal{E}_2(f)$.

If $x\notin\mathcal{E}_2(f)$ lies in $\mathrm{Supp}(R_{f^s})$, then Lemma~\ref{lem: classical izumi type inequality} and the scaling property of log canonical thresholds give
\begin{equation}\tag{$\ast$}\label{eq: uniform-lct-estimate}
\mathrm{lct}_x\left(\mathbf{P}^2;\frac{R_{f^s}}{q^s-1}\right)=(q^s-1)\mathrm{lct}_x(\mathbf{P}^2;R_{f^s})\geq\frac{q^s-1}{\mathrm{ord}_x(R_{f^s})}\geq\frac{q^s-1}{C\rho^s}.
\end{equation}
If $\mathcal{E}_2(f)=\varnothing$, then, since $\rho<q$, the pair is klt on $\mathrm{Supp}(R_{f^s})$ for all sufficiently large $s$ by \eqref{eq: uniform-lct-estimate}, and it is log smooth outside the support. Hence the pair is globally klt for all sufficiently large $s$. So we are left with the only situation $\mathcal{E}_2(f)\neq\varnothing$.

\smallskip
\noindent\textbf{Step 3.}
Throughout this step, whenever $f$ is replaced by $f^r$, the symbols $f$ and $q$ are reused for $f^r$ and $q^r$, respectively.
By Proposition~\ref{prop: FJ-exceptional}(ii), the finite set $\mathcal{E}_2(f)$ is totally invariant. After replacing $f$ by an iterate, we may assume that every point $p\in\mathcal{E}_2(f)$ is $f^{-1}$-invariant.
By miracle flatness, $f$ is finite flat of global degree $q^2$, so the fixed-point germ $f:(\mathbf{P}^2,p)\to(\mathbf{P}^2,p)$ also has local topological degree $q^2$. The germ is superattracting by Proposition~\ref{prop: FJ-exceptional}(ii), and $c_\infty(p,f)=q$.

For each $p\in\mathcal{E}_2(f)$, the differential $df_p$ is nilpotent; hence, by the chain rule, $d(f^2)_p=0$. Theorem~\ref{thm: balanced-local} gives a primitive divisorial valuation $v_p$ centered at $p$ satisfying $(f^2)_*v_p=q^2v_p$. Replacing $f$ by $f^2$ and $q$ by $q^2$, we may therefore assume simultaneously that, for every $p\in\mathcal{E}_2(f)$, $df_p=0$ and
$$ f_*v_p=qv_p. $$
After changing $C$ and $\rho$, Proposition~\ref{prop: FJ-exceptional}(iii) shows that \eqref{eq: uniform-lct-estimate} remains valid for the resulting endomorphism.

For each $p\in\mathcal{E}_2(f)$, Proposition~\ref{prop: crepant-descent}, applied to $f$ and $v_p$, gives an integer $s_p\geq 1$ such that $(\mathbf{P}^2,\frac{R_{f^{s_p}}}{q^{s_p}-1})$ is log canonical in a neighborhood of $p$. Choose a common multiple $s$ of the integers $s_p$ so large that $(q^s-1)/(C\rho^s)>1$.
Recall from Definition~\ref{defn: log discrepancy quasi-monomial valuations} that $A_{\mathbf{P}^2}$ denotes the log discrepancy function of $(\mathbf{P}^2,0)$.
A reformulation of \cite[Proposition~5.20]{KM98} gives, for every divisorial valuation $w$ and every $m\geq 1$,
$$ A_{\mathbf{P}^2}((f^m)_*w)=A_{\mathbf{P}^2}(w)+w(R_{f^m}). $$
Fix $p\in\mathcal{E}_2(f)$ and let $w$ be a divisorial valuation whose center on $\mathbf{P}^2$ has closure containing $p$. The local log canonicity for $f^{s_p}$ gives
$$ A_{\mathbf{P}^2}((f^{s_p})_*w)\leq q^{s_p}A_{\mathbf{P}^2}(w). $$
Since $f$ fixes $p$, every successive pushforward under $f^{s_p}$ has center whose closure contains $p$. Iterating the preceding inequality $s/s_p$ times therefore gives
$$ A_{\mathbf{P}^2}((f^s)_*w)\leq q^sA_{\mathbf{P}^2}(w). $$
Equivalently,
$$
A_{\left(\mathbf{P}^2,R_{f^s}/(q^s-1)\right)}(w)
 =\frac{q^sA_{\mathbf{P}^2}(w)-A_{\mathbf{P}^2}((f^s)_*w)}{q^s-1}
 \geq 0.
$$
Thus $(\mathbf{P}^2,R_{f^s}/(q^s-1))$ is log canonical at every point of $\mathcal{E}_2(f)$.

By \eqref{eq: uniform-lct-estimate}, the pair is klt at every point of $\mathrm{Supp}(R_{f^s})\setminus\mathcal{E}_2(f)$, and it is log smooth outside the support. Therefore $(\mathbf{P}^2,R_{f^s}/(q^s-1))$ is log canonical on $\mathbf{P}^2$ and klt on $\mathbf{P}^2\setminus\mathcal{E}_2(f)$. This finishes the proof.
\end{proof}

\begin{remark}
The preceding local-to-global argument can be simplified by directly applying \cite[Theorem~1.16(2)]{CZ}.
\end{remark}

\section{The Jacobian exceptional set}

Let $f:\mathbf{P}^n\to\mathbf{P}^n$ be a $q$-polarized endomorphism, where $q>1$. Locally, the ramification divisor $R_f$ is defined by the Jacobian determinant of $f$. Thus $\mathrm{ord}_x(R_f)$ is the order of vanishing of a local Jacobian determinant at $x$. The chain rule gives
$$ R_{f^{s+t}}=R_{f^s}+(f^s)^*R_{f^t} $$
for all $s,t\geq1$.

A function $\varphi:\mathbf{P}^n\to\mathbb{R}$ is \emph{Zariski upper semicontinuous} if $\{x:\varphi(x)\geq a\}$ is Zariski closed for every $a\in\mathbb{R}$. On the smooth variety $\mathbf{P}^n$, the critical divisor of $f^s$ in the notation of \cite{Par11} is the ramification divisor $R_{f^s}$. Following \cite[\S~4.1 and the proof of Theorem~3.5]{Par11}, set
$$ \mu_s(x):=n+\mathrm{ord}_x(R_{f^s}),\qquad s\geq1. $$
The family $(\mu_s)_{s\geq1}$ is the \emph{Jacobian cocycle} of $f$. Each $\mu_s$ is Zariski upper semicontinuous, and \cite[Proposition~4.1]{Par11} gives
$$ \mu_{s+t}(x)\leq\mu_s(x)\mu_t(f^s(x)) $$
for all $s,t\geq1$. This property is called \emph{submultiplicativity}. A closed subset $Z\subseteq\mathbf{P}^n$ is \emph{totally invariant} if $f^{-1}(Z)=Z$; since $f$ is surjective, this also implies $f(Z)=Z$.

By \cite[\S~4.1]{Par11}, the limit
$$ \mu_\infty(x):=\lim_{s\to\infty}\mu_s(x)^{1/s} $$
exists for every $x\in\mathbf{P}^n$. Moreover, \cite[Theorem~4.3]{Par11} identifies the set introduced in the Introduction with
$$ \mathcal{E}_{\mathrm{Jac}}(f)=\{x\in\mathbf{P}^n:\mu_\infty(x)=q\}. $$

\begin{proof}[Proof of Theorem~\ref{thm: fixed Jacobian exceptional set}]
The algebraicity and total invariance of $\mathcal{E}_{\mathrm{Jac}}(f)$ follow from \cite[Theorem~4.3 and Corollary~4.5]{Par11}. Taking $s=1$ in the defining inequalities gives $\mathcal{E}_{\mathrm{Jac}}(f)\subseteq\mathrm{Supp}(R_f)$, so this set is proper. By \cite[Theorem~4.7]{Par11}, there exist $C>0$ and $0\leq\rho<q$ such that $\mu_s(x)\leq C\rho^s$ for every $s\geq1$ and every $x\notin\mathcal{E}_{\mathrm{Jac}}(f)$.

If $x\notin\mathcal{E}_{\mathrm{Jac}}(f)$ lies in $\mathrm{Supp}(R_{f^s})$, the scaling property of log canonical thresholds and Lemma~\ref{lem: classical izumi type inequality} give
$$ \mathrm{lct}_x\left(\mathbf{P}^n;\frac{R_{f^s}}{q^s-1}\right)=(q^s-1)\mathrm{lct}_x(\mathbf{P}^n;R_{f^s})\geq\frac{q^s-1}{\mathrm{ord}_x(R_{f^s})}\geq\frac{q^s-1}{C\rho^s}. $$
For all sufficiently large $s$, the right-hand side is greater than one independently of $x$. Hence the pair is klt at every such point, and it is also klt outside $\mathrm{Supp}(R_{f^s})$. This proves the last assertion.
\end{proof}

\begin{remark}
The preceding description exhibits a gap in the exponential growth of the Jacobian cocycle. If $x\in\mathcal{E}_{\mathrm{Jac}}(f)$, then $\mathrm{ord}_x(R_{f^s})\geq\delta_xq^s$ for some $\delta_x>0$ and every $s\geq1$. Since $\mathrm{ord}_x(R_{f^s})\leq\deg R_{f^s}=(n+1)(q^s-1)$, it follows that $\mu_\infty(x)=q$. On the complement of $\mathcal{E}_{\mathrm{Jac}}(f)$, \cite[Theorem~4.7]{Par11} gives $\mu_s(x)\leq C\rho^s$ for some $C>0$ and $\rho<q$, and hence $\mu_\infty(x)\leq\rho<q$. Thus $\mathcal{E}_{\mathrm{Jac}}(f)$ is precisely the locus of maximal exponential growth.
\end{remark}

\begin{remark}
The functions $\eta_s(x):=1+\mathrm{lct}_x(\mathbf{P}^n;R_{f^s})^{-1}$ also form an analytic submultiplicative cocycle in the sense of \cite[Definition~1.1 and Theorem~1.2]{Din09}. Its reverse maximal-growth locus is a proper algebraic subset of $\mathbf{P}^n$ contained in $\mathcal{E}_{\mathrm{Jac}}(f)$ and is forward invariant. The construction does not give total invariance, so this locus does not provide a fixed exceptional set containing the non-klt loci of the normalized ramification pairs.
\end{remark}

\section{Further questions}

\subsection{Higher-dimensional eigenvaluations and equivariant extraction}

The surface proof combines the divisorial eigenvaluation supplied by Theorem~\ref{thm: balanced-local} with the equivariant extraction of Corollary~\ref{cor: equivariant-extraction-surface}. Neither step has a direct higher-dimensional analogue.

\begin{question}\label{ques: higher-dimensional eigenvaluations}
Let $n\geq3$ and $F:(\mathbb{C}^n,0)\to(\mathbb{C}^n,0)$ be a finite superattracting germ satisfying $d(F)=c_\infty(F)^n=q^n$, where $q>1$ is an integer. Here $d(F)$ is the local degree and $c_\infty(F)$ the asymptotic attraction rate. Does there exist an integer $m\geq1$ and a quasi-monomial valuation $v$ centered at the origin such that
$$ (F^m)_*v=q^m v? $$
Can $v$ always be chosen to be primitive divisorial? More generally, if $f:\mathbf{P}^n\to\mathbf{P}^n$ is $q$-polarized and $Z\subseteq\mathcal{E}_{\mathrm{Jac}}(f)$ is a totally invariant irreducible subvariety of codimension at least two, can one find an integer $s\geq1$ and a primitive divisorial valuation $v$ centered at $Z$ such that
$$ (f^s)_*v=q^s v? $$
\end{question}

On a smooth surface, every primitive divisorial valuation centered at a closed point is the unique Rees valuation of a suitable simple complete ideal \cite[Appendix~5]{ZS}. The analogous special complete ideal may have several Rees valuations in higher dimensions \cite[Theorem~6.8]{HK}.

\begin{question}\label{ques: equivariant extraction without unique Rees valuation}
Let $X$ be a smooth projective variety of dimension at least three, $f:X\to X$ a $q$-polarized endomorphism, and $p\in X$ a point satisfying $f^{-1}(p)=\{p\}$. Suppose that $v$ is a primitive divisorial valuation centered at $p$ and $f_*v=qv$. Without assuming that $v$ is the unique Rees valuation of an $\mathfrak{m}_p$-primary ideal, can $v$ be extracted on a normal $\mathbb{Q}$-factorial projective model carrying a finite $q$-polarized lift $g$ of $f$, possibly together with additional exceptional prime divisors, so that the divisor $E$ realizing $v$ satisfies $g^{-1}(E)=E$ and $g^*E=qE$? More generally, can one find an $\mathfrak{m}_p$-primary ideal whose Rees valuations contain $v$ and are permuted by $q^{-1}f_*$, and then extract them simultaneously so that the reduced exceptional divisor is totally invariant?
\end{question}

A relative version of Question~\ref{ques: equivariant extraction without unique Rees valuation} would be needed for the last part of Question~\ref{ques: higher-dimensional eigenvaluations}. Even then, extraction alone would not resolve the higher-dimensional adjunction problem.

\subsection{Effective iteration bounds}

Theorem~\ref{thm: main projective plane} gives an iterate depending on $f$, while \cite[Example~1.5]{LZ26} shows that the first iterate need not suffice.

\begin{question}\label{ques: effective iteration bound}
For a fixed integer $q>1$, is there an effective bound $N(q)$ such that every $q$-polarized endomorphism $f:\mathbf{P}^2\to\mathbf{P}^2$ admits an integer $1\leq s\leq N(q)$ for which $(\mathbf{P}^2,R_{f^s}/(q^s-1))$ is log canonical? Can $N(q)$ be chosen independently of $q$?
\end{question}

\subsection{Normalized-volume dynamics}

We now recall the volume of a valuation \cite{ELS03} and the normalized volume of a klt singularity \cite{Li18}; see also \cite[Definitions~2.1 and~2.3]{LLX}.

\begin{definition}\label{defn: normalized volumes}
Let $(X,\Delta)$ be an $n$-dimensional klt pair, let $x\in X$ be a closed point, and let $v\in\mathrm{Val}_{X,x}$. The volume of $v$ is
$$ \mathrm{vol}_{X,x}(v):=\limsup_{m\to\infty}\frac{\ell_{\mathcal{O}_{X,x}}\bigl(\mathcal{O}_{X,x}/\mathfrak{a}_m(v)\bigr)}{m^n/n!}. $$
The normalized volume of $v$ with respect to $(X,\Delta)$ is
$$
\widehat{\mathrm{vol}}_{(X,\Delta),x}(v):=
\begin{cases}
A_{(X,\Delta)}(v)^n\mathrm{vol}_{X,x}(v), & \text{if }A_{(X,\Delta)}(v)<+\infty,\\
+\infty, & \text{if }A_{(X,\Delta)}(v)=+\infty.
\end{cases}
$$
The normalized volume of the singularity $x\in(X,\Delta)$ is
$$ \widehat{\mathrm{vol}}(x,X,\Delta):=\inf_{v\in\mathrm{Val}_{X,x}}\widehat{\mathrm{vol}}_{(X,\Delta),x}(v). $$ By definition, we have $\widehat{\mathrm{vol}}_{(X,\Delta),x}(\lambda v)=\widehat{\mathrm{vol}}_{(X,\Delta),x}(v)$ for every $\lambda\in\mathbb{R}_{>0}$.
A valuation attaining the above infimum is called a \emph{normalized volume minimizer}.
\end{definition}

Normalized volume minimizers play a central role in the local K-stability theory of klt singularities. The existence, quasi-monomiality, and uniqueness results developed in \cite{Blu18,Xu20,XZ21,BLQ24}, together with related finite-generation results in the global stability-threshold setting \cite{LXZ22}, culminate in the stable degeneration theorem \cite[Theorem~1.2(1)]{XZ25}, which associates to every klt singularity a K-semistable log Fano cone degeneration.

\medskip

Let $X$ be a smooth projective variety of dimension $n$, let $f: X\to X$ be a finite $q$-polarized endomorphism, and let $x\in X$ be a totally invariant point, so that $f^{-1}(x)=\{x\}$. The global degree of $f$ is $q^n$. Since $f$ is finite flat and its fiber over $x$ is supported at $x$, the local degree of the germ $f:(X,x)\to(X,x)$ is also $q^n$.

Put $R=\mathcal{O}_{X,x}$ and $\mathfrak{m}=\mathfrak{m}_x$. For a valuation $v\in\mathrm{Val}_{X,x}$ normalized by $v(\mathfrak{m})=1$, recall that
$$ (f_*v)(h):=v(f^*h),\qquad c(f,v):=v(f^*\mathfrak{m}),\qquad f_\bullet v:=\frac{1}{c(f,v)}f_*v. $$
Here $h\in K(X)^*$. Then $f_\bullet v(\mathfrak{m})=1$, and $(f^s)_\bullet v=f_\bullet^s v$ for every $s\geq 1$. In this subsection, we abbreviate
$\widehat{\mathrm{vol}}_x(v):=\widehat{\mathrm{vol}}_{(X,0),x}(v)$.

The equality criterion in Lemma~\ref{lem: local-volume} gives the following exact description.

\begin{proposition}\label{prop: normalized-volume defect}
In the above setting, let $v\in\mathrm{Val}_{X,x}$ be divisorial and normalized by $v(\mathfrak{m})=1$. For every integer $s\geq 1$, write $v_s:=f_\bullet^s v$ and $c_s:=c(f^s,v)$, and set
$$ D_s(v):=\frac{q^{ns}\mathrm{vol}_{X,x}(v_s)}{c_s^n\mathrm{vol}_{X,x}(v)}. $$
\begin{enumerate}
    \item Let $\mathfrak{b}^{(s,v)}_t:=(f^s)^*\mathfrak{a}_t(v_s)R$, and $\mathfrak{c}^{(s,v)}_t:=\mathfrak{a}_{c_s t}(v)$. Then $D_s(v)\geq 1$, and equality holds if and only if $$ \widetilde{\mathfrak{b}}^{(s,v)}_\bullet=\mathfrak{c}^{(s,v)}_\bullet. $$
    \item $D_{s+r}(v)=D_s(v)D_r(v_s)$ for all integers $s,r\geq 1$. In particular, after setting $D_0(v)=1$, the sequence $D_s(v)$ is nondecreasing.
    \item We have $$ \frac{\widehat{\mathrm{vol}}_x(v_s)}{\widehat{\mathrm{vol}}_x(v)}=D_s(v)\left(\frac{A_X((f^s)_*v)}{q^sA_X(v)}\right)^n. $$
    Moreover, $$ \widehat{\mathrm{vol}}_x(v_s)\leq\widehat{\mathrm{vol}}_x(v)\quad\Longleftrightarrow\quad A_{(X,R_{f^s}/(q^s-1))}(v)\geq\frac{q^s}{q^s-1}A_X(v)\left(1-D_s(v)^{-1/n}\right)\geq 0. $$
    In particular, normalized volume nonincrease implies $A_{(X,R_{f^s}/(q^s-1))}(v)\geq 0$, and strict decrease implies $A_{(X,R_{f^s}/(q^s-1))}(v)>0$.
\end{enumerate}
\end{proposition}

\begin{proof}
Apply Lemma~\ref{lem: local-volume} to $(f^s)^*:R\to R$, whose rank is $q^{ns}$, with $\mu=v_s$ and $c=c_s$. This gives $D_s(v)\geq 1$. Since $v$ is divisorial and $\mathrm{vol}_{X,x}(v)>0$, $\mathfrak{c}^{(s,v)}_\bullet$ is saturated by \cite[Lemma~3.20]{BLQ24}. The equality criterion in Lemma~\ref{lem: local-volume} gives the stated equality condition. This proves $(1)$.

Since $v$ is divisorial, every $v_s$ is divisorial. Moreover, $v_{s+r}=f_\bullet^r v_s$ and
$$ c(f^{s+r},v)=c(f^s,v)\cdot \frac{1}{c(f^s,v)} v((f^s)^*((f^r)^*\mathfrak{m}))=c_s\,((f^s)_\bullet v)((f^r)^*\mathfrak{m})=c_s\,c(f^r,v_s). $$
Substitution into the definition of $D_{s+r}(v)$ gives $D_{s+r}(v)=D_s(v)D_r(v_s)$. Since every factor $D_r(v_s)$ is at least one, the monotonicity assertion follows. This proves $(2)$.

Since normalized volume is invariant under positive rescaling and $(f^s)_*v=c_s v_s$, we have
\begin{equation}\label{eq: fact of local volume}
    \frac{\widehat{\mathrm{vol}}_x(v_s)}{\widehat{\mathrm{vol}}_x(v)}=\frac{A_X((f^s)_*v)^n\mathrm{vol}_{X,x}(v_s)}{c_s^nA_X(v)^n\mathrm{vol}_{X,x}(v)}=D_s(v)\left(\frac{A_X((f^s)_*v)}{q^sA_X(v)}\right)^n,
\end{equation}
which is the asserted factorization.
By \cite[Proposition~5.20]{KM98},
$$ A_X((f^s)_*v)=A_X(v)+v(R_{f^s}). $$
This gives
\begin{equation}\label{eq: log dis formula}
    A_{(X,R_{f^s}/(q^s-1))}(v)=A_X(v)-\frac{v(R_{f^s})}{q^s-1}=\frac{q^sA_X(v)-A_X((f^s)_*v)}{q^s-1}.
\end{equation}
Combining \eqref{eq: log dis formula} and \eqref{eq: fact of local volume} and solving for $A_{(X,R_{f^s}/(q^s-1))}(v)$ proves the stated equivalence. The remaining consequences follow from $D_s(v)\geq 1$.
\end{proof}

Proposition~\ref{prop: normalized-volume defect} indicates that if there exists a uniform $s$ such that $\widehat{\mathrm{vol}}_x(f_\bullet^s v)\leq\widehat{\mathrm{vol}}_x(v)$ for all divisorial valuations $v\in \mathrm{Val}_{X,x}$ normalized by $v(\mathfrak{m})=1$, then $(X,\frac{R_{f^s}}{q^s-1})$ is lc near $x$. Its exact criterion shows that such a statement requires simultaneous control of the discrepancy term and the saturation defect.

\begin{question}\label{ques: normalized-volume dynamics}
Is there a useful description of the behavior of the sequence $\{\widehat{\mathrm{vol}}_x(f_\bullet^s v)\}_{s\geq 0}$?
\end{question}

\end{document}